\documentclass[preprint]{elsarticle}
\usepackage{amssymb, amsmath, amsthm}
\usepackage[bookmarksnumbered, colorlinks, plainpages]{hyperref}
\usepackage{comment}
\usepackage{tikz}
\usepackage{array, float}
\usepackage[ruled,vlined]{algorithm2e}
\usepackage{booktabs}
\usepackage{bbm}
\usepackage{adjustbox}
\usepackage{orcidlink}

\newtheorem{theorem}{Theorem}
\newtheorem{lemma}{Lemma}
\newtheorem{corollary}{Corollary}
\newtheorem{proposition}{Proposition}

\theoremstyle{definition}
\newtheorem{example}{Example}
\newtheorem{remark}{Remark}

\newcommand{\cf}{\mathrm{cf}}
\newcommand{\avg}{\mathrm{avg}}
\newcommand{\AlgA}{\mathcal{A}}

\begin{document}
\begin{frontmatter}

\title{Bounds on the Complete Forcing Numbers of Graphs}

\author[1]{Javad B. Ebrahimi \orcidlink{0000-0003-3414-5040}}
\ead{javad.ebrahimi@sharif.ir}

\author[1]{Aref Namayandeh \orcidlink{0009-0003-7777-0599}}
\ead{arefnemayandeh@gmail.com}

\author[1]{Elahe Tohidi \orcidlink{0009-0004-5788-6583}}
\ead{elahetohidi2017@gmail.com}
\date{}
\affiliation[1]{organization={Department of Mathematical Sciences},
addressline={Sharif University of Technology},
city={Tehran},
postcode={11155-9415},
country={Iran}
}

\begin{abstract}
    A complete forcing set of a graph is a subset $S$ of its edges with the following property. For any perfect matching $M$, $M$ is the unique perfect matching containing the intersection of $M$ and $S$. The size of a smallest complete forcing set of a graph is called its complete forcing number. In this paper, we propose a polynomial-time algorithm that takes a graph and an ordering of its vertices as input, and outputs a complete forcing set. We analyze the size of the output for a carefully chosen ordering to derive an upper bound on the complete forcing numbers of graphs in terms of other graph-theoretical parameters, such as the degeneracy and the spectral radius of the underlying graph. We compare our results with previous findings in the field. For the class of edge-transitive graphs, we provide an explicit lower bound on the complete forcing number. This result is applied in particular to hypercube graphs and Cartesian powers of even cycles.
\end{abstract}


\end{frontmatter}

\section{Introduction}
For some basic definitions and notations of graph theory, we follow the textbook~\cite{west2001}. Let $G$ be a graph with the vertex set $V(G)$ and the edge set $E(G)$. A vertex is called \textit{incident} to an edge if it is one of the endvertices of that edge. Two edges that do not share an endvertex are called \textit{disjoint}. Two vertices $u$ and $v$ are called \textit{adjacent} if $uv$ is an edge. A sequence $v_0,v_1,\ldots,v_m$, where $v_i$'s are distinct vertices and every pair of consecutive vertices in this sequence is adjacent is called a \textit{path} of \textit{length} $m$. A graph is called \textit{connected} if there exists at least one path between each pair of distinct vertices. If such paths are unique, the graph is called a \textit{tree}. In a connected graph, the length of a shortest path between $u$ and $v$ is called the distance between them. 
	
A set of pairwise disjoint edges is called a \textit{matching} of $G$. A matching \textit{covers} a vertex $v$ if $v$ is incident to exactly one of the edges in that matching. A \textit{perfect matching} of $G$ is a matching that covers all vertices. A graph that has at least one perfect matching is called a \textit{matchable} graph. Let $M$ be a perfect matching of $G$. A subset $S \subseteq E(G)$ is called a \textit{forcing set} for $M$ if $M$ is the unique perfect matching containing $S$.

A subset $S$ of $E(G)$ is called a \textit{complete forcing set} of $G$ if for every perfect matching $M$, the subset $S \cap M$ forms a forcing set of $M$. Note that $S \cap M$ need not be a minimum forcing set. The size of a smallest complete forcing set of a graph $G$ is called the \textit{complete forcing number} of $G$, and is denoted by $\cf(G)$. The main challenge is to calculate the exact value, or upper and lower bounds, for $\cf(G)$ given an input graph $G$.

The following example helps us to review some of the notions defined above. We also use the same graph $H$ in Figure~\ref{fig:example} as an illustrative example for some of the upcoming definitions and results.

\begin{example} \label{example:forcing}
Consider the graph $H$ in Figure~\ref{fig:example}.
\begin{figure}
    \centering
    \begin{tikzpicture}[scale= 1.5, every node/.style={circle, draw, fill=white, inner sep=2pt, minimum size=15pt}]
        \node (a) at (0, 1) {a};
        \node (b) at (0, 0) {b};
        \node (c) at (1, 1) {c};
        \node (d) at (1, 0) {d};
        \node (e) at (2, 1) {e};
        \node (f) at (2, 0) {f};
        \node (g) at (3, 1) {g};
        \node (h) at (3, 0) {h};

        \draw (a) -- (b);
        \draw (a) -- (h);
        \draw (b) -- (c);
        \draw (b) -- (d);
        \draw (c) -- (d);
        \draw (c) -- (e);
        \draw (c) -- (f);
        \draw (d) -- (e);
        \draw (d) -- (f);
        \draw (e) -- (f);
        \draw (e) -- (g);
        \draw (f) -- (g);
        \draw (g) -- (h);
    \end{tikzpicture}
    \caption{Graph \( H \)}
    \label{fig:example}
\end{figure} 
 This graph contains seven different perfect matchings. For instance,  $M_1=\{ab,cd,ef,gh\}$ is a perfect matching with forcing number $1$, since $ef$ is contained in only one perfect matching $M_1$ of $H$. On the other hand, $M_2=\{ah,bc,de,fg\}$ is a perfect matching with forcing number $2$, and a forcing set $\{bc, fg\}$. The reason is that each edge of $M_2$ belongs to at least one perfect matching other than $M_2$. Thus, $M_2$ has no forcing set of size $1$. Also, $M_2$ is the unique perfect matching containing $\{bc, fg\}$.
 
It is easy to check that the set $S = \{cf, fd, de, ec, ah, ef\}$ forms a complete forcing set for $H$. Therefore, $\cf(H) \leq 6$. By a computer search we find that the exact answer matches the upper bound.
\end{example}

\subsection{Prior Works}
\begin{table}
\setlength{\tabcolsep}{10pt} 
\begin{adjustbox}{max width=\textwidth}
\begin{tabular}{|m{1.9cm}|m{1.5cm}|m{2.1cm}|m{4.2cm}|}
\hline
\textbf{Authors} & \textbf{Reference} & \textbf{Graph Type} & \textbf{Result Type} \\ 
\hline
He and Zhang&  
\cite{he2024complete}(2024)& (4,6)-Fullerenes & Lower bound is presented. \\ 
\hline
 He and Zhang&
\cite{he2023completemultipartite}(2023)& Complete and Almost-Complete Multipartite Graphs & Exact solution is found. \\ 
\hline
Wei et al.&
\cite{wei2023complete}(2023)& Catacondensed Phenylene Systems & Exact solution is found. \\
\hline
He and Zhang& 
\cite{he2023complete}(2023)& All Graphs & Upper bound is presented. \\ 
\hline
He and Zhang &
\cite{he2023complete}(2023)& Wheels and Cylinders & Exact solution is found.\\ 
\hline
He and Zhang &
\cite{he2022hexagonal}(2022)& Hexagonal Systems & Upper and lower bounds are presented . \\ \hline
He and Zhang &
\cite{he2021hexagonal}(2021)& Hexagonal Systems & A complete forcing set and a lower bound are presented. \\ 
\hline
Liu et al.&\cite{liu2021complete}(2021)&Spiro Hexagonal
Systems& Exact solution for several classes of spiro hexagonal systems is found.\\
\hline
Chang et al.&
\cite{chang2021complete}(2021)& Rectangular Polynominoes &  Exact solution is found. \\ 
\hline
Lui et al.&\cite{liu2016complete}(2016)& Phenylene Systems& Exact solution is found.\\ \hline
Xu et al. &\cite{xu2016coronoid}(2016)& Primitive Coronoid & Exact solution is found. \\ 
\hline
Chan et al.&\cite{chan2015linear}(2015)&Catacondensed Hexagonal
System& A linear relationship between the complete forcing number and the Clar number is presented. \\ \hline
Xu et al.&\cite{xu2015complete}(2015)& Catacondensed Hexagonal Systems & Exact solution is found. \\ 
\hline
\end{tabular}
  \end{adjustbox}
  \medskip
\caption{Summary of research papers on complete forcing numbers}
\label{tab:priorworks}
\vspace{-60pt} 
\end{table}
Inspired by the application of chemistry in studying the Kekulé structures of benzenoids (benzenoid graphs), in 2015, Xu et al. \cite{xu2015complete} defined the notions of \textit{complete forcing set} and \textit{complete forcing number} of a graph $G$. In the same year, Chan et al. provided a linear-time algorithm in \cite{chan2015linear}, for explicitly computing the complete forcing number of a catacondensed hexagonal graph. For a definition of these graphs, see~\cite{xu2015complete}. 

In~\cite{xu2015complete}, Xu et al. studied the problem of computing the complete forcing numbers of graphs using the notion of \textit{nice cycle} which turned out to be useful, and is extensively used in the literature of complete forcing set. In Section~\ref{sec:preliminaries}, we define this notion, and the way it is used for the problem of computing the complete forcing number.

The result of \cite{xu2015complete} is used to compute the complete forcing number of catacondensed hexagonal systems. Inspired by this article, many subsequent papers have computed complete forcing numbers for other families of graphs that have significance and applications in chemistry. (For instance, see \cite{xu2016coronoid, liu2016complete, liu2021complete, he2021hexagonal, he2022hexagonal, wei2023complete}.)


In 2021, Chang et al.~\cite{chang2021complete} calculated the complete forcing number of the grid graph, i.e., the Cartesian product of two paths. 

In 2023, He and Zhang~\cite{he2023complete} proved that for any graph $G$, $\cf(G) \le 2(|E(G)| - |V(G)| + \omega(G))$, where $\omega(G)$ is the number of connected components of $G$. The quantity $c(G) := |E(G)| - |V(G)| + \omega(G)$ is called the cyclomatic number of $G$. To the best of our knowledge, prior to this work, this result was the only non-trivial explicit upper bound in terms of the number of vertices and edges that can be applied to all graphs. In the same work, He and Zhang found the exact solution for the wheel graph $W_n$, as well as for the cylinder graphs. (The wheels $W_n$ consist of a cycle of length $n-1$ with another vertex that is joined to all the vertices of the cycle. The cylinder graph is the Cartesian product of a path with a cycle).

In 2023, He and Zhang calculated the complete forcing numbers of complete and almost-complete multipartite graphs in \cite{he2023completemultipartite}. In the same work, they also provided an explicit upper bound on this parameter in terms of the number of edges and the maximum degree. 

In 2024, they found a lower bound on the complete forcing number of $(4,6)$-fullerenes in \cite{he2024complete} and conjectured that their bound is exact. They proved their conjecture for certain $(4, 6)$-fullerenes. 

For a comprehensive survey on this topic, see Zhang et al. \cite{zhang2025forcing}. We summarize the prior works on complete forcing numbers in Table~\ref{tab:priorworks}.

\subsection{Our Contribution}
The main results of this paper can be divided into two parts; namely, results about upper bounds and results regarding lower bounds on the complete forcing number of edge-transitive graphs. In this section, we describe the results and state them in a formal manner. Definitions of key terms such as degeneracy, spectral radius, and edge-transitivity are given in Section~\ref{sec:preliminaries}.
 
 \begin{itemize}
     \item{\textbf{Upper bound:}}
 We first present a polynomial-time algorithm that takes an arbitrary graph $G$ and an arbitrary ordering $\pi$ of its vertices as input, and outputs a complete forcing set $S$ for $G$. By analyzing the size of the output set $S$ for a carefully chosen ordering $\pi$, we are able to find upper bounds on the complete forcing number of the graph in terms of other graph-theoretical parameters, such as the degeneracy and the spectral radius of $G$. The main results of this part are the following theorems.

\begin{theorem} \label{thm:spectral-radius}  
 Let $G$ be a graph with spectral radius $\rho(G)$. Then,   
\[\cf(G) \le \left( 1 - \frac{1}{\rho(G)} \right) |E(G)|.\]  
\end{theorem}

\begin{theorem} \label{thm:degeneracy}  
Let $G$ be a $d$-degenerate graph with maximum degree $\Delta$. Then,  \[\cf(G) \le \left( 1 - \frac{1}{2\sqrt{d\Delta} - d} \right) |E(G)|.\] 
\end{theorem}

As a corollary of these results, we obtain upper bounds on the complete forcing number of planar and outerplanar graphs and Cartesian product of any number of arbitrary trees.

\item{\textbf{Lower bound:}}
When $G$ is an edge-transitive graph, we derive a lower bound on the complete forcing number in terms of the maximum forcing number of perfect matchings in $G$. More precisely, we prove the following theorem.

\begin{theorem} \label{thm:edge-transitive}
Let $G$ be an edge-transitive graph, and let $F(G)$ denote the maximum forcing number of $G$. Then, 
\[\cf(G) \ge \frac{2|E(G)|}{|V(G)|}F(G).\]
\end{theorem}

In Section \ref{sec:main}, we present several special cases and corollaries of this result. Most notably, we derive lower bounds on the complete forcing number of hypercube graphs and the Cartesian power $C_{n}^{k}$ for even $n$ and any $k\geq 2$. For hypercube graphs, we also provide an upper bound that is close to the lower bound we obtain.
 \end{itemize}
 
\subsection{Organization of the Paper}
The rest of this paper is organized as follows. To better describe the proofs of the main results, we present some prerequisite definitions and propositions in Section~\ref{sec:preliminaries}.  Section~\ref{sec:main} contains the proofs of the main theorems of this paper. This section also includes several special cases and corollaries of the main theorems.
 We finally conclude the paper in Section \ref{conclusion}.
\section{Preliminaries} \label{sec:preliminaries}
Let $G=(V, E)$ be a graph. Throughout this paper, we assume that all graphs are simple, undirected, connected and matchable. The size of a smallest forcing set of perfect matching $M$ is called \textit{forcing number} of $M$, and is denoted by $f(G,M)$. The \textit{maximum forcing number} of $G$ is the maximum value among the forcing numbers of all perfect matchings of $G$, and is denoted by $F(G)$. 
	
A subgraph $H$ of $G$ is called \textit{nice} if $G \setminus V(H)$ contains a perfect matching, where $G \setminus V(H)$ denotes the induced subgraph of $G$ on the vertex set $V(G) \setminus V(H)$. Note that an even cycle $C$ of $G$ is nice if and only if $C$ is the symmetric difference of two distinct perfect matchings, $M_1$ and $M_2$. We refer to $M_1 \cap C$ and $M_2 \cap C$ as the two \textit{frames} of $C$.
	
Xu et al. established an equivalent condition for a subset of edges of a graph to be a complete forcing set.
	
\begin{proposition}[Theorem 2.7 of \cite{xu2015complete}] \label{prop:nice}
Let $G$ be a graph with edge set $E(G)$. Then, the set $S \subseteq E(G)$
is a complete forcing set of $G$ if and only if for any nice cycle $C$ of $G$, the set $S$ intersects each frame of $C$.
\end{proposition}

The \textit{neighborhood} of a vertex $v$ in $G$, denoted by $N_G(v)$ (or simply $N(v)$), is the subset of vertices in $G$ that are adjacent to $v$. The \textit{closed neighborhood} of $v$, denoted by $N[v]$, is defined as $N(v) \cup \{v\}$. The size of $N(v)$, called the degree of $v$, is denoted by $d_v$. The \textit{$2$-degree} of a vertex $v$, denoted by $t_v$, is the sum of the degrees of the vertices adjacent to \( v \). i.e.,
\(t_v := \sum_{u \in N(v)} d_u.\)
	
The \textit{average $2$-degree} of $v$ is defined as $\avg(v):=\frac{t_v}{d_v}$ if $d_v\neq 0$. When $d_v = 0$, we define $\avg(v)$ to be $0$.
	
Let $S$ be a nonempty proper subset of $V(G)$. The set of all edges of $G$ having exactly one endvertex in $S$ is denoted by $\partial_G(S)$ (or simply $\partial(S)$). For a vertex $v \in V(G)$, we write $\partial(v)$ in place of $\partial(\{v\})$. It is clear that $\partial(v)$ is the set of all edges incident to $v$.
	
An \textit{automorphism} of a graph $G$ is a one-to-one mapping from the vertex set of $G$ to itself that preserves the adjacency of the vertices. The set of automorphisms of $G$ under the composition forms a well-known automorphism group denoted by $\mathrm{Aut}(G)$, which naturally acts on the set of edges. If $\mathrm{Aut}(G)$ transitively acts on $E(G)$, then the graph is called \textit{edge-transitive}. For a comprehensive review of the notions of automorphism group and its actions, the interested reader is referred to \cite{godsil2001algebraic}.
	
The \textit{Cartesian product} $G \times H$ of two graphs $G$ and $H$ is a graph with the vertex set being the Cartesian product of their vertex sets, i.e, $V(G)\times V(H)$, and two vertices $(u, v)$ and $(u', v')$ are adjacent if and only if either $u = u'$ and $vv' \in E(H)$, or $v = v'$ and $uu' \in E(G)$. For a graph $G$, we denote the $k$-fold Cartesian product of $G$ by $G^k$, and call it the \textit{$k$-th Cartesian power} of $G$. The $n$-dimensional \textit{hypercube} graph, denoted by $Q_n$, is $n$-th Cartesian power of the complete graph with two vertices (i.e., $Q_n={K_2}^n$).
	
For a nonnegative integer $d$, a graph $G$ is called \textit{$d$-degenerate} if for every induced subgraph $H$ of $G$, $\delta (H)\leq d$, where $\delta(H)$ denotes the minimum degree of the vertices of $H$. The \textit{degeneracy} of $G$ is the smallest such $d$.

Let $A$ be an $ n \times n $ matrix. A scalar $ \lambda $ is called an \textit{eigenvalue} of $ A $ if there exists a non-zero vector $ \vec{v} \in \mathbbm{R}^n $ such that
$A \vec{v} = \lambda \vec{v}$,
where $ \vec{v} $ is called an eigenvector corresponding to the eigenvalue $ \lambda $.

Let  $V(G) = \{v_1,v_2,\ldots,v_n\}$. The \textit{adjacency matrix} of $G$ is $A(G) = [a_{ij}]_{n \times n}$, where $a_{ij} = 1$ if $v_iv_j \in E(G)$ and $a_{ij} = 0$ otherwise. The \textit{spectral radius} of graph $G$ is the largest eigenvalue of $A(G)$, and is denoted by $\rho(G)$.

The following proposition is a special case of a well-known theorem called Interlacing Theorem.
\begin{proposition}[Special Case of Interlacing Theorem, Lemma 8.6.3 of \cite{godsil2001algebraic}] \label{prop:interlacing}
    Let $G$ be a graph with spectral radius $\rho(G)$. If $H$ is an induced subgraph of a graph $G$, then $\rho(H) \le \rho(G)$.
\end{proposition}

\section{Main results} \label{sec:main}
We present our results in two parts. In the first part, we derive upper bounds on the complete forcing number of graphs in terms of other graph-theoretical parameters, such as the degeneracy and the spectral radius of graphs.

The second part consists of finding a lower bound on the complete forcing numbers of edge-transitive graphs. Such graphs contain several important families of graphs, such as hypercube graphs and Cartesian powers of even cycles.

\subsection{Upper Bound on Complete Forcing Number}
In this section, we first present a polynomial-time algorithm for constructing a complete forcing set for any input graph and an ordering of its vertices. After proving that the output of this algorithm forms a complete forcing set, we try to analyze the size of the output. This analysis consists of upper bounding the size of the output in terms of other graph-theoretical parameters. We must mention that these upper bounds are also computable in polynomial-time.

\begin{algorithm}
\caption{An Algorithm for Generating a Complete Forcing Set} \label{alg}
\KwIn{Graph $G$ with an ordering $\pi$ on its vertices}
\KwOut{A complete forcing set $S$ of $G$}

Let $B$ be an empty set\;
$G_1 \gets G$\;
$i \gets 1$\;

\While{$G_i$ is not empty}{
    Let $v_i \in V(G_i)$ be the vertex with the minimum index in $\pi$\;
    $B_i \gets \partial_{G_i}(v_i)$\;
    $B \gets B \cup B_i$\;
    $G_{i+1} \gets G_i \setminus N_{G_i}[v_i]$\;
    $i \gets i + 1$\;
}

$S \gets E(G) \setminus B$\;
\Return $S$\;
\end{algorithm}
The algorithm is presented in Algorithm~\ref{alg}. In the $i$-th step of this algorithm, a vertex $v_i$ is selected, where $v_i$ is the first vertex in the ordering $\pi$ that still exists in the remaining graph $G_i$. The edges incident to $v_i$ are added to the set $B$, and both $v_i$ and its neighbors are removed from $G_i$ to obtain the graph $G_{i+1}$.
For a given graph \( G \) and vertex ordering \( \pi \), the output of Algorithm~\ref{alg} is denoted by $\AlgA(G, \pi)$.

\begin{example} \label{example:algorithm}
In graph \(H\), shown in Figure~\ref{fig:example}, the output of the algorithm for the ordering \(\pi = (a, b, c, \ldots, h)\) is:
\[\AlgA(H, \pi) = \{bc, bd, hg, de, df, fe, fg, eg\}.\]
In the first step, vertex \(a\) is selected, and the edges \(\{ab, ah\}\) are added to the set \(B\). Then, in the second step, vertex $c$ is selected ($b$ has been removed), and the edges \(\{ce, cf, cd\}\) are added to \(B\), and the algorithm will halt after this step. Thus, the remaining $8$ edges form the complete forcing set for $H$.
\end{example}

\begin{lemma} \label{lem:algorithm}
    Let $G$ be a graph, and $\pi$ be any ordering on its vertices. The set $S = \AlgA(G, \pi)$ forms a complete forcing set of $G$.
\end{lemma}

\begin{proof}
    It is clear that the algorithm eventually terminates and returns a set \( S \subseteq E(G)\). Let \( C \) be an arbitrary nice cycle of $G$. According to Proposition~\ref{prop:nice}, it is sufficient to prove that the set \( S \) contains an edge from each frame of \( C \). Let \( I = \{v_1, \ldots, v_t\} \) represent the sequence of vertices of $G$, selected by the algorithm according to the ordering \( \pi \). If \( C \) does not include any edge from \( B = E(G)\setminus S\), then the correctness of the claim is obvious.

Therefore, assume \( C \) intersects the set \( B \). Let \( A = V(C) \cap I \), where $V(C)$ is the set of the vertices of $C$. By the construction, it is clear that each edge of $B$ has exactly one endvertex in \( I \), so \( A \) is non-empty. Suppose \( j \) is the smallest index such that \( v_j \in A \).

Let \( x \) and \( y \) be the two vertices in $V(C)$ that are adjacent to \( v_j \) in $C$. We consider the following two cases for \( x \).

\begin{itemize}
    \item{\textbf{Case 1:}} The vertex \( x \) was removed in an earlier step, say step \( k\). This means, \( x = v_k \) or \( x \in N(v_k)\). Since \( k < j \), and we assumed \( j \) is the smallest index, \( x = v_k \) is impossible. Therefore, \( x \) is one of the neighbors of \( v_k \), and the edge \( v_jx \) belongs to the set \( S \). 
    
    Let \( z \) be the other neighbor of \( x \) in \( C \) that is distinct from \( v_j \). For the same reason mentioned above, \( z \) cannot be \( v_k \). Hence, according to the algorithm, the edge \( zx \) also belongs to \( S \). Consequently, the set $S$ contains two consecutive edges $v_jx$ and $xz$ of $C$, meaning that $S$ intersects both frames of $C$.
    
    \item{\textbf{Case 2:}} The vertex \( x \) was not removed before the \( j \)-th step. Similarly, if $y$ was removed before the $j$-th step, then the claim holds. So, assume that \( y \) was also not removed before step \( j \). We know that the vertices \( x \) and \( y \), the neighbors of \( v_j \), are present in graph \( G_j \) at this step, and the edges \( v_jx \) and \( v_jy \) are added to the set \( B \).
    
    Let $z$ be the neighbor of $x$ in $C$ that is distinct from $v_j$. It follows that \( xz \in S \). Additionally, \( z \neq y \) because a cycle of length $3$ cannot be a nice cycle. Similarly, if \( u \) is another neighbor of \( y \), then \( yu \in S \). (It is possible that \( u = z \)). Therefore, the edges \( yu \) and \( xz \), which belong to two different frames of \( C \), are in the set \( S \). Thus, the claim holds once again.
\end{itemize}

\end{proof}

\begin{remark}
In the reference \cite{he2023completemultipartite}, Lemma 2.3, it was proven that in any matchable graph $G$, if $W$ is a subset of the vertices such that the distance between every pair of the elements of $W$ in $G$ is at least $3$, then the set $S:=E(G)\setminus \partial (W)$ has the property that it intersects each frame of any nice cycle of $G$. This, in particular, implies that $S$ is a complete forcing set for $G$. This idea gives an upper bound on the complete forcing number. The best upper bound we get from this idea is when $W$ contains vertices with pairwise distance at least $3$ while $S$ is as small as possible. Equivalently, the best bound from this approach is to make $|\partial(W)|$ as large as possible. 

Now, we claim that our result, for any graph $G$, is at least as good as the one in \cite{he2023completemultipartite}, and we present an example of a graph for which our result is strictly better. In addition, analyzing $|\partial (W)|$ in terms of other parameters of the graph is not known. Even the computational complexity of finding the set $W$ with the mentioned properties, which maximizes $|\partial(W)|$ is not known.  

To prove the claim that our result is never weaker than the one in \cite{he2023completemultipartite}, we proceed as follows. Let $W$ be a set of the vertices such that the distance between each pair of distinct vertices in $W$ is at least $3$, and $|\partial(W)|$ is the maximum among all possible choices of such subsets. Now, consider an ordering $\pi$ on the vertex set such that the first $|W|$ vertices are all elements of $W$ of any arbitrary ordering, followed by the rest of the vertices again in any arbitrary ordering. Then, the output of the algorithm $\AlgA(G,\pi)$, will be contained in the set $E(G)\setminus \partial(W)$.

As an example for which our result gives a strictly better bound, consider the graph $L$ shown in Figure~\ref{fig:goodness}. Note that in this graph, the distance between every pair of the vertices is at most $2$. Hence, $W$ can only be a singleton. Thus, the maximum possible size of $\partial(W)$ is equal to $\Delta(L) = 5$. Thus, the best upper bound we get, has size $|E(G)\setminus \partial(W)|=13$. Now, let $\pi$ be the ordering with $a$ and $e$ being the first two elements. Running Algorithm~\ref{alg} with this ordering yields $|\AlgA(L, \pi)|=12$, which is strictly smaller.
\end{remark}
\begin{figure}
\centering
\begin{tikzpicture}[scale= 1.3, every node/.style={circle, draw, fill=white, inner sep=2pt, minimum size=15pt}]
    \node (a) at (1, 1) {a};
    \node (b) at (0, 0) {b};
    \node (c) at (2, 1) {c};
    \node (d) at (1, 0) {d};
    \node (e) at (1, -1) {e};
    \node (f) at (2, 0) {f};
    \node (g) at (2, -1) {g};
    \node (h) at (3, 0) {h};

    \draw (a) -- (c);
    \draw (a) -- (b);
    \draw (a) -- (d);
    \draw (a) -- (f);
    \draw (a) -- (h);
    \draw (c) -- (b);
    \draw (c) -- (d);
    \draw (c) -- (f);
    \draw (c) -- (h);
    \draw (e) -- (b);
    \draw (e) -- (g);
    \draw (e) -- (d);
    \draw (e) -- (f);
    \draw (e) -- (h);
    \draw (g) -- (h);
    \draw (g) -- (b);
    \draw (g) -- (d);
    \draw (g) -- (f);
\end{tikzpicture}
\caption{Graph \( L \)}
\label{fig:goodness}
\end{figure} 
The next lemma is the key component to prove Theorems \ref{thm:spectral-radius} and \ref{thm:degeneracy}. 
\begin{lemma} \label{lem:subgraph-average-degree}
    Let \( G \) be a nonempty graph and $t$ be a real number such that for every induced subgraph \( H \) of \( G \), there exists a vertex \( v \in V(H) \) satisfying 
    $ \avg_H(v) \le t. $
    Then, 
    \[ \cf(G) \le \left(1 - \frac{1}{t}\right) |E(G)|.  \]
\end{lemma}

\begin{proof}
We begin by describing an ordering $\pi$ of the vertices of $G$. Let $G_1:=G$. According to the assumption of the lemma, there exists a vertex $v_1$ in $G_1$ such that $\avg_{G_1}(v_1) \le t$. Let this vertex $v_1$ be the first vertex in the ordering $\pi$. Next, let $G_2$ be an induced subgraph of $G_1$ obtained by deleting  $N_{G_1}[v_1]$. Note that at this point, the vertex $v_1$ is placed in $\pi$ but its neighbors have been deleted from $G_1$ without placement in $\pi$. At the end, we will describe how to place these vertices in the ordering. 

Again, by the assumption of the lemma, there exists a vertex $v_2$ in $G_2$ with $\avg_{G_2}(v_2)\leq t$. We set this vertex $v_2$ as the second vertex in the ordering $\pi$. We continue this process iteratively, selecting a vertex in each successive subgraph $G_i$ with $\avg_{G_i}(v_i) \le t$, and defining the next vertex in the ordering $\pi$, until no vertices remain. Finally, we complete the ordering $\pi = (v_1, v_2, \ldots, v_n)$ by arbitrarily ordering the vertices that are deleted during the process. 

Let $S := \AlgA(G, \pi)$ be the output of Algorithm~\ref{alg} for the input $(G,\pi)$. Assume the algorithm proceeds for \( k \) steps, selecting the vertices \( \{v_1, \dots, v_k\} \). For \( i = 1, \dots, k \), let \( B_i := \partial_{G_i}(v_i) \) be defined as in the algorithm. Let \( A_i := \bigcup \limits_{u:u \in N_{G_i}(v_i)} \partial_{G_i}(u) \) be the set of all edges incident to a neighbor of \( v_i \). From the algorithm, by Lemma~\ref{lem:algorithm}, we know that the sets $A_i \setminus B_i$ are pairwise disjoint, and $S = \bigcup_{i = 1}^k (A_i \setminus B_i)$ is a complete forcing set for $G$.

By the construction of $\pi$, for each \( i = 1, \dots, k \) we have
\[\frac{|A_i|}{|B_i|} \le \avg_{G_i}(v_i) \le t.\]
Therefore, 
\[\frac{|A_i| - |B_i|}{|B_i|} \le t - 1 \implies \frac{|S|}{|E(G)\setminus S|} = \frac{\sum_{i = 1}^k (|A_i| - |B_i|)}{\sum_{i = 1}^k |B_i|} \le t - 1.\]
This inequality implies that 
\[\cf(G) \le |S| \le \left(\frac{t-1}{t}\right) |E(G)| = \left(1 - \frac{1}{t}\right)|E(G)|.\]

\end{proof}

\subsubsection{Proof of Theorem 1.}

Using Lemma~\ref{lem:subgraph-average-degree} and the concept of the spectral radius of a graph, we can find a new upper bound on the complete forcing number of all graphs. First, we state a useful proposition.

\begin{proposition}[Theorem 4 of \cite{yu2004spectral}] \label{prop:spectral-avg}
Let $G$ be a graph with degree sequence $d_1, d_2, \dots, d_n$. Then,
\[\rho(G) \ge \sqrt{\frac{t_1^2 + t_2^2 + \dots + t_n^2}{d_1^2 + d_2^2 + \dots + d_n^2}},\]
where $t_i$ denotes the 2-degree of vertex $i$.
\end{proposition}

\begin{lemma} \label{lem:avg-degree}
Let $G$ be a graph and $\rho(G)$ be its spectral radius. Then, there exists a vertex $v \in V(G)$ such that 
\[\avg(v) \le \rho(G).\]
\end{lemma}

\begin{proof}
Let $v$ be the vertex that minimizes $\avg(v)$. It is obvious that 
\[\avg(v)^2 \le \frac{\sum_{u \in V(G)}t_u^2}{\sum_{u \in V(G)}d_u^2} \le \rho(G)^2,\]
which implies
$\avg(v) \le \rho(G).$ The second inequality follows directly from Proposition~\ref{prop:spectral-avg}.
\end{proof}


Next, we prove Theorem~\ref{thm:spectral-radius}.
\begin{proof}[Proof of Theorem~\ref{thm:spectral-radius}]
First, notice that for every induced subgraph $H$ of $G$, we have $\rho(H)\leq \rho(G)$, according to Proposition~\ref{prop:interlacing}. By Lemma~\ref{lem:avg-degree}, we conclude that if $v\in V(H)$, then $\avg_H(v) \leq \rho(H)\leq \rho(G)$. Now, we apply the result of Lemma~\ref{lem:subgraph-average-degree} with $t=\rho(G)$. Note that the assumption of this lemma is satisfied, hence we conclude the theorem.
\end{proof}

\begin{example}
By calculating the eigenvalues of the graph $H$ in Figure~\ref{fig:example}, we obtain \(\rho(H) = \frac{\sqrt{5} + 5}{2}\) as the spectral radius of \(H\). Thus, Theorem~\ref{thm:spectral-radius} yields
\[\cf(H) \le \left( 1- \frac{2}{\sqrt{5} + 5} \right) \cdot 13 \approx 9.4 \implies \cf(H) \le 9. \]
It is easy to check that $H$ is $3$-degenerate. By applying the Theorem~\ref{thm:degeneracy}, we obtain
\[\cf(H) \le \big( 1 - \frac{1}{2\sqrt{3\cdot 4} - 3} \big) \cdot 13 \approx 9.7 \implies \cf(H) \le 9.\]
\end{example}

We now present some corollaries of this theorem.
\begin{corollary}
    Let $G$ be a graph with $n$ vertices and $m$ edges. Then, 
    \[\cf(G)\leq \left(1-\frac{1}{\sqrt{2m-n+1}}\right)m.\]
\end{corollary}

\begin{proof}
 In \cite{YUAN1988135}, Theorem 1, it was shown that for every connected graph $G$ with $n$ vertices and $m$ edges, we have $\rho(G)\leq \sqrt{2m-n+1}$. Since we only consider connected graphs in this paper, the proof follows immediately from Theorem~\ref{thm:spectral-radius}.
\end{proof}

\begin{remark}
In \cite{he2023complete}, Theorem 2.4, it was proven that for every connected graph $G$ with $n$ vertices and $m$ edges, $\cf(G)\leq 2(m-n+1)$. In the same paper, the authors characterized all the graphs for which equality holds. Note that this bound is only non-trivial when $m<2n-2$. Otherwise, the trivial upper bound $\cf(G)\leq m$ is stronger. In contrast, the upper bound in terms of $n$ and $m$ given in the above corollary is always strictly less than $m$, hence a stronger upper bound.
\end{remark}

\begin{corollary}\label{cor:planar-spectral-radius}
    If $G$ is a planar graph with maximum degree $\Delta \ge 2$, then,
    \[\cf(G) \le \left(1 - \frac{1}{\sqrt{8\Delta - 16} + 2\sqrt{3}} \right) |E(G)|.\]
\end{corollary}

\begin{proof}
    In Theorem 5.3 of \cite{dvovrak2010spectral}, it was shown that for every planar graph with maximum degree $\Delta \ge 2$, we have $\rho(G) \le \sqrt{8\Delta - 16} + 2\sqrt{3}$. Therefore,
    this corollary follows from Theorem~\ref{thm:spectral-radius} and the upper bound on the spectral radius of the planar graphs. 
\end{proof}

\begin{remark}
Notice that the same argument as in the previous result can be applied to obtain an upper bound on the complete forcing number, when we have an exact or an upper bound on the spectral radius. In reference \cite{dvovrak2010spectral}, upper bounds on the spectral radius of graphs that are embedded on higher genus surfaces have been derived. Thus, one may find an upper bound on the complete forcing numbers of graphs in terms of the maximum degree, the genus, and the number of the edges of the graph. 
\end{remark}


\subsubsection{Proof of Theorem 2.}

We start by stating one useful lemma.
\begin{lemma} \label{lem:d-degenerate}
    Let $G$ be a nonempty, $d$-degenerate graph with maximum degree $\Delta$. Then there exists a vertex $v \in V(G)$ such that 
   $\avg(v) \le 2\sqrt{d\Delta} - d.$
\end{lemma}
\begin{proof}
    For any positive real number $0 < x \le \Delta$, let $S(x) \subseteq V(G)$ be the subset of vertices with degree at least $x$. Let $H = G[S(x)]$ be the induced subgraph of $G$ on the vertex set $S(x)$. By the definition of $d$-degenerate graphs, there exists a vertex $v \in S(x)$ such that $d_H(v) \le d$.

Let $y := d_H(v)$. By the definition of $H$, $v$ has $y$ neighbors in $G$ with degree greater than or equal to $x$, and by our selection, $y \le d$. The other $d_G(v) - y$ neighbors of $v$ have degree less than $x$. Thus, we have
\[\avg(v) \le \frac{y\cdot \Delta + (d_G(v) - y)x}{d_G(v)} = \frac{y(\Delta - x)}{d_G(v)} + x. \]
Since $y \le d$, and $d_G(v) \ge x$, we have
\[\avg(v) \le \frac{d(\Delta - x)}{x} + x = \frac{d\Delta}{x} + x - d.\]
Now we can optimize the value of $x$ to get the best possible bound. It is not hard to see that the minimum is attained in $x = \sqrt{d\Delta}$. Hence,
$\avg(v) \le 2\sqrt{d\Delta} - d.$
\end{proof}


We now prove Theorem~\ref{thm:degeneracy}.
\begin{proof}[Proof of Theorem~\ref{thm:degeneracy}]
    It is clear that every induced subgraph of $G$ is a $d$-degenerate graph with maximum degree $\Delta$. Therefore, by Lemma~\ref{lem:d-degenerate}, the result of Lemma~\ref{lem:subgraph-average-degree} holds for $t = 2\sqrt{d\Delta} - d$.
\end{proof}

\begin{remark}
In Theorem \ref{thm:spectral-radius} we obtained an upper bound on $\cf(G)$ in terms of the spectral radius of $G$. One might ask that if there is an upper bound on $\rho(G)$ in terms of the degeneracy of $G$, then the combined inequalities will give an upper bound on $\cf(G)$ for $d$-degenerate graphs. This is indeed the case, and as pointed out in \cite{hayes2006simple}, the spectral radius is related to the degeneracy of $G$ via the inequality $\rho(G)\leq \sqrt{4d(\Delta-d)}$, if $G$ is $d$-degenerate with maximum degree $\Delta$. 

Consequently, we can derive from the result of Theorem \ref{thm:spectral-radius} that \[\cf(G) \leq \left(1-\frac{1}{\sqrt{4d(\Delta-d)}}\right)|E(G)|.\]

However, this upper bound is weaker than the bound given in Theorem \ref{thm:degeneracy}. 

\end{remark}

We now present some corollaries of this theorem.
\begin{corollary} \label{cor:upperbound-delta}
Let $G$ be a graph with maximum degree $\Delta$. Then, for $\Delta$-regular graphs \[\cf(G) \le \left(1 - \frac{1}{\Delta}\right)|E(G)|,\]
and for non-regular graphs
\[\cf(G) \le \left(1 - \frac{1}{2\sqrt{\Delta(\Delta-1)}-\Delta +1 }\right)|E(G)|.\] 

\end{corollary}
\begin{proof}
Note that the degeneracy of a $\Delta$-regular graph is $\Delta$, and for a non-regular connected graph, its degeneracy is at most $\Delta-1$. Thus, the assertion follows immediately from Theorem~\ref{thm:degeneracy}.
\end{proof}

\begin{remark}
In \cite{he2023completemultipartite}, Corollary 2.5, it was shown that $\cf(G)\leq |E(G)|-\Delta(G)$. This bound, in general, is incomparable with the upper bound in Corollary~\ref{cor:upperbound-delta}. However, for any constant $\Delta$, one can observe that, the result in the corollary is stronger than the result of \cite{he2023completemultipartite} by a multiplicative factor smaller than $1$, which depends only on $\Delta$.
\end{remark}

\begin{corollary} \label{cor:productoftrees}
    Let $T_1, T_2, \ldots, T_m$ be $m$ trees. Suppose that $T_i$ has $n_i>1$ vertices. Then, we have
    \[
    \cf(T_1\times \cdots\times T_m) \leq \left(1-\frac{1}{2\sqrt{m\Delta}-m}\right)\Big(\sum\limits_{i=1}^{m} \big((n_i-1)\prod \limits_{j:j\neq i}n_j\big)\Big),
    \]
    where $\Delta$ is the maximum degree of $T_1 \times \cdots \times T_m$.
\end{corollary}

\begin{proof}
    It was shown in \cite{bickle2010cores}, and mentioned in Section 3 of \cite{hickingbotham2023structural}, that if $G_1$ and $G_2$ are $r_1$- and $r_2$-degenerate graphs, respectively, then their Cartesian product is $(r_1+r_2)$-degenerate. Since every tree is $1$-degenerate, by a simple induction we can observe that the Cartesian product of $m$ trees is $m$-degenerate. Moreover, since the number of edges of a tree is one less than the number of its vertices, it is straightforward to prove that the number of edges of $T_1\times \cdots\times T_m$ is equal to $\Big(\sum\limits_{i=1}^{m} \big((n_i-1)\prod \limits_{j:j\neq i}n_j\big)\Big)$. Finally, we apply Theorem \ref{thm:degeneracy} to conclude the claim.
\end{proof}

\begin{remark}
In \cite{chang2021complete}, it was shown that $\cf(P_m\times P_n)=\lfloor \frac{n}{2}\rfloor(m-1) + \lfloor \frac{m}{2}\rfloor(n-1)$, which is approximately $mn$. The result of  Corollary \ref{cor:productoftrees} gives the upper bound approximately $1.45mn$, which is weaker by a $1.45$ multiplicative factor. However, the result of Corollary \ref{cor:productoftrees} can be used for any number of trees of any kind. For instance, if $G$ is Cartesian product of $m$ paths each of which of length at least $2$, then $\Delta=2m$, and the upper bound we get is $\cf(G)\leq \left(1-\frac{1}{1.83m}\right)|E(G)|$.    
\end{remark}

\begin{corollary} \label{cor:planar-degeneracy}
    If $G$ is a planar graph, then, $\cf(G) \leq \left( 1 - \frac{1}{2\sqrt{5\Delta} - 5}\right) |E(G)|$. 
\end{corollary}
\begin{proof}
    If $G$ is a planar graph, then its degeneracy is at most $5$ (see for example \cite{west2001}, Theorem 6.1.23.). The claim follows immediately from Theorem~\ref{thm:degeneracy}.
\end{proof}
A graph $G$ is called \textit{outerplanar} if it is planar, and for some plane drawing of $G$, all the vertices belong to the boundary of a single face.
\begin{corollary} \label{cor:outerplanar}
     If $G$ is an outerplanar graph, then,
     $\cf(G) \le \left(1 - \frac{1}{2\sqrt{2\Delta }-2 } \right) |E(G)|.$
\end{corollary}

\begin{proof}
    If $G$ is an outerplanar graph, then it is $2$-degenerate (see for example \cite{west2001}, Proposition 6.1.20.). The claim follows immediately from Theorem~\ref{thm:degeneracy}.
\end{proof}

\begin{remark}
For the planar graphs, the upper bound provided in Corollary~\ref{cor:planar-spectral-radius} is better than the bound in Corollary~\ref{cor:planar-degeneracy}. For a general outerplanar graph, to the best of our knowledge, the best upper bound on the spectral radius in terms of the maximum degree is the one which holds for all the planar graphs. For these graphs, Corollary~\ref{cor:outerplanar} gives a better upper bound compared to the bound obtained by Theorem~\ref{thm:spectral-radius}. This observation implies that the results of both theorems are incomparable for general graphs.
\end{remark}

\subsection{Lower Bound on Complete Forcing Number}
 In this part, we derive a lower bound on the complete forcing number of the edge-transitive graphs in terms of the maximum forcing number of them. Also, we use this result to find a lower bound on the complete forcing number of edge-transitive regular bipartite graphs only in terms of the number of the vertices and the vertex degree. 
 Before we prove this theorem, we need the following lemma:
	\begin{lemma}\label{lem:lower-forcing-number}
		Let $M_1, M_2, \dots, M_m$ be some perfect matchings of a graph $G$. Suppose that each edge of $G$ belongs to at most $k$ of these matchings. Let $f(M_i)$ denote the forcing number of $M_i$. Then, we have
		\[\cf(G) \ge \frac{1}{k}\sum_{i = 1}^{m}f(M_i).\]
	\end{lemma}
	
	\begin{proof}
		Let $S$ be a complete forcing set for $G$ of size $\cf(G)$. From the definition, we know that for each $i = 1, \dots, m$, $S \cap M_i$ is a forcing set for $M_i$. Thus, $|S \cap M_i| \ge f(M_i)$. Therefore, we have
		\[k \cdot \cf(G) \ge \sum_{i = 1}^m |S \cap M_i| \ge \sum_{i = 1}^{m} f(M_i).\]
		
		The first inequality holds because each edge of $G$, and specifically $S$, belongs to at most $k$ of the matchings $M_i$. Therefore, each of the $|S| = \cf(G)$ edges is counted at most $k$ times in the sum $\sum_{i = 1}^m |S \cap M_i|$.
	\end{proof}

In order to apply the above lemma, one needs to find a set of perfect matchings such that the forcing number of each one is known. Also, for the bound to be large, one needs to take the perfect matchings such that $k$ is not too large. The key point is that it is not clear how we can have the best of both worlds; namely, we have a set of perfect matchings with large forcing number, and at the same time $k$ is relatively small. The next result guarantees that for the edge-transitive graphs, both of these properties can be achieved. 


Now we are ready to prove Theorem~\ref{thm:edge-transitive}.
\begin{proof}[Proof of Theorem~\ref{thm:edge-transitive}]
	Let $A$ be the automorphism group of the graph $G$. Since $G$ is edge-transitive, $A$ acts transitively on the edge set of $G$. That is, for every pair of the edges $e_1,e_2$ of $G$, there exists an automorphism $g\in A$ such that $g(e_1)=e_2$. By the action of the group $A$ on $E$, the image of any subset $E'$ of $E(G)$ under the element $h\in A$ is the set $h(E')$ which is isomorphic to $E'$ as a subgraph. Let us call the subset $h(E')$ a \emph{copy} of $E'$. Note that if $e_1$ belongs to a copy $h(E')$ of $E'$, then $e_2$ belongs to the copy $gh(E')$ of $E'$. Therefore, due to the symmetry caused by the edge transitivity of $G$, each edge of $G$ appears in the same number of the copies of $E'$. Let $k$ be the number of the copies of $E'$ including any fixed edge $e$. We first find the value of $k$ in terms of other parameters of the problem. The total number of edges covered by a copy of $E'$, taking multiplicity into account, is equal to $|E'| \cdot |A|$. Alternatively, this quantity is equal to $|E|\cdot k$ since each edge appears in $k$ copies. Thus, we have $k=\frac{|A|\cdot |E'|}{|E|}.$
	
	Now, set $E'$ to be a perfect matching of $G$ with forcing number $F(G)$. Since the copies of $E'$ are isomorphic to $E'$, each copy $g(E')$ is also a perfect matching with forcing number $F(G)$. Now, we can apply Lemma~\ref{lem:lower-forcing-number} to conclude that
	\[\cf(G) \ge \frac{1}{k} \sum_{g\in A} {f(g(E'))} = \frac{|E|}{|A|\cdot|E'|} |A|\cdot F(G) =\frac{2|E(G)|}{|V(G)|} F(G).\]
	The first inequality is due to the previous lemma. The first equality follows from substituting the value of $k$, and also the fact that all the copies of $E'$ have the same forcing number $F(G)$. Finally, the last equality is simply because each perfect matching has $\frac{|V(G)|}{2}$ edges.
\end{proof}

We now apply the previous result to the case of graph \( Q_n \). The following proposition gives us a lower bound on complete forcing number of all edge-transitive regular bipartite graphs.
\begin{proposition}[Theorem 3 of \cite{adams2004forced}] \label{prop:regular-bipartite}
    For any \( k \)-regular bipartite graph \( G \) with \( \frac{n}{2} \) vertices in each part, we have
    \[F(G) \ge \left(1 - \frac{\log(2\operatorname{e})}{\log(k)}\right)\frac{n}{2},\]
    where ``\( \operatorname{e}\)'' is the base of the natural logarithm.
\end{proposition}
The following corollary follows immediately from Theorem~\ref{thm:edge-transitive} and the above proposition.
\begin{corollary}
    Let \( G \) be an edge-transitive \( k \)-regular bipartite graph with $n$ vertices. Then,
    \[\cf(G) \ge \left(1 - \frac{\log(2 \operatorname{e})}{\log(k)}\right) \frac{nk}{2}.\]
\end{corollary}
It is well-known that hypercube graph $Q_n$ is an edge-transitive $n$-regular bipartite graph. Therefore, the following corollary holds. 
\begin{corollary}\label{cor:Qn}
For every integer $n\geq 4$, 
\[\left(1 - \frac{\log(2 \operatorname{e})}{\log(n)}\right) n\cdot 2^{n-1} \le \cf(Q_n) \le n \cdot 2^{n-1} - 5 \cdot 2^{n-3}.\]
\end{corollary}
\begin{proof}
The previous corollary immediately implies the lower bound. For the upper bound, we utilize Proposition~\ref{prop:nice}. According to this proposition, we need to find a subset of edges $S$ such that it intersects each frame of every nice cycle of $Q_n$. Such a set, of size $n \cdot 2^{n-1} - 5 \cdot 2^{n-3}$, has been obtained in \cite{ebrahimi2020fractional}, within the proof of Theorem 35. This will complete the proof.
\end{proof}

The next application of Theorem \ref{thm:edge-transitive} is the following result. 
\begin{corollary}\label{cor:cn}
    Let $n\ge 4$ be an even number and $k \ge 2$. Then, $\cf(C_n^k) \ge \frac{k}{2}n^k$.
\end{corollary}

To prove this corollary, we need the following lemma.
\begin{lemma}\label{lem:F}
For any even number $n\ge 4$ and any integer $k\geq 2$, we have $F(C_n^k) \ge \frac{n^k}{4}$. \footnote{This result was stated for any $k\geq 1$ in \cite{ebrahimi2023maximum} while the proof only works for $k\geq 2$.}
\end{lemma}

\begin{proof}
First, notice that we can define the cycle graph $C_n$ with vertex set $V(C_n)=\mathbbm{Z}_n$ and edge set $E(C_n)=\big\{ ij : i,j\in \mathbbm{Z}_n, \text{ and } i-j = \pm 1 \pmod{n} \big\}$. Then, the $k$-fold Cartesian product of $C_n$, denoted by $C_n^k$, has vertex set $\mathbbm{Z}_{n}^k$  such that two vertices $(a_1,\dots,a_k)$ and $(a'_1,\dots,a'_k)$ are adjacent in $C_n^k$ if and only if there exists a unique index $j\in\{1, 2, \ldots, k\}$ with $a_j\neq a'_j$, and in this case $|a_{j}-a_{j}^{\prime}|=1 \pmod{n}$.

Consider the following perfect matching.
\[M \;=\;\bigl\{\{(2a,\bar{b}),(2a+1,\bar{b})\}:a\in\{1,\ldots,\frac{n}{2}\},\;\bar{b}\in\mathbbm{Z}_{n}^{\,k-1}\bigr\}.\] For each \(a,b\in \{1,2,\ldots, \frac{n}{2}\}\) and \(\bar{c}\in \mathbbm{Z}_{n}^{k-2}\), the four vertices
\[
  (2a,2b,\bar{c}),(2a+1,2b,\bar{c}),(2a+1,2b+1,\bar{c}),(2a,2b+1,\bar{c})
\]
form a 4-cycle, and exactly one of its frames is contained in \(M\). Call such a cycle an \textit{alternating cycle}. Figure~\ref{fig:F} illustrates the perfect matching $M$ and the proposed alternating cycles for $n=4$ and $k=2$.

Let $S$ be any forcing set of $M$. As observed by Pachter and Kim in \cite{pachter1998forcing}, and also stated formally by Riddle in Proposition 1 of \cite{riddle2002minimum}, the set $S$ contains an edge from every alternating cycle $C$. Otherwise, by switching the edges in the frames of $C$, we obtain another perfect matching $M'$ that also contains the set $S$, contradicting the assumption that $S$ is a forcing set.

Since there are $\frac{n^k}{4}$ such disjoint cycles, and $S$ contains at least one edge from each of them, it follows that
\[
  F(C^k_{n}) \ge |S| \ge \frac{n^k}{4},
\]
as required. 
\end{proof}

\begin{figure} \label{fig:F}
    \centering
    \begin{tikzpicture}[scale= 0.8, every node/.style={circle, draw, fill=black, inner sep=1.5pt, minimum size=2pt}, every path/.style={thick}]
			
			\node (a1) at (0, 4) {};
			\node (a2) at (1, 4) {};
			\node (a3) at (2, 4) {};
			\node (a4) at (3, 4) {};
			
			\node (b1) at (0, 3) {};
			\node (b2) at (1, 3) {};
			\node (b3) at (2, 3) {};
			\node (b4) at (3, 3) {};
			
			\node (c1) at (0, 2) {};
			\node (c2) at (1, 2) {};
			\node (c3) at (2, 2) {};
			\node (c4) at (3, 2) {};
			
			\node (d1) at (0, 1) {};
			\node (d2) at (1, 1) {};
			\node (d3) at (2, 1) {};
			\node (d4) at (3, 1) {};
			
			\draw[draw=blue] (a1) -- (b1);
			\draw[draw=blue] (a2) -- (b2);
			\draw[draw=blue] (a3) -- (b3);
			\draw[draw=blue] (a4) -- (b4);
			
			\draw[dotted, shorten >=8pt] (a1) -- (0,5);
			\draw[dotted, shorten >=8pt] (a2) -- (1,5);
			\draw[dotted, shorten >=8pt] (a3) -- (2,5);
			\draw[dotted, shorten >=8pt] (a4) -- (3,5);
			
			\draw (b1) -- (c1);
			\draw (b2) -- (c2);
			\draw (b3) -- (c3);
			\draw (b4) -- (c4);
			
			\draw[draw=blue] (c1) -- (d1);
			\draw[draw=blue] (c2) -- (d2);
			\draw[draw=blue] (c3) -- (d3);
			\draw[draw=blue] (c4) -- (d4);
			
			\draw[dotted, shorten >=8pt] (d1) -- (0,0);
			\draw[dotted, shorten >=8pt] (d2) -- (1,0);
			\draw[dotted, shorten >=8pt] (d3) -- (2,0);
			\draw[dotted, shorten >=8pt] (d4) -- (3,0);
			
			\draw (a1) -- (a2);
			\draw (a2) -- (a3);
			\draw (a3) -- (a4);
			
			\draw[dotted, shorten >=8pt] (a1) -- (-1,4);
			\draw[dotted, shorten >=8pt] (a4) -- (4,4);
			
			\draw (b1) -- (b2);
			\draw (b2) -- (b3);
			\draw (b3) -- (b4);
			
			\draw[dotted, shorten >=8pt] (b1) -- (-1,3);
			\draw[dotted, shorten >=8pt] (b4) -- (4,3);
			
			\draw (c1) -- (c2);
			\draw (c2) -- (c3);
			\draw (c3) -- (c4);
			
			\draw[dotted, shorten >=8pt] (c1) -- (-1,2);
			\draw[dotted, shorten >=8pt] (c4) -- (4,2);
			
			\draw (d1) -- (d2);
			\draw (d2) -- (d3);
			\draw (d3) -- (d4);
			
			\draw[dotted, shorten >=8pt] (d1) -- (-1,1);
			\draw[dotted, shorten >=8pt] (d4) -- (4,1);
			
			\draw[red, thin] (0.5, 3.5) circle (0.4);
			\draw[red, thin] (2.5, 3.5) circle (0.4);
			\draw[red, thin] (0.5, 1.5) circle (0.4);
			\draw[red, thin] (2.5, 1.5) circle (0.4);
		\end{tikzpicture}
    \caption{The graph $C_4^2$. Dotted edges in the same direction are connected to each other. The perfect matching $M$ is drawn in blue edges, and the red circles induce the alternating cycles in Lemma~\ref{lem:F}.}
    \label{fig:my_label}
\end{figure}
Using the previous proposition we conclude the following corollary of Theorem~\ref{thm:edge-transitive}.

\begin{proof}[Proof of Corollary \ref{cor:cn}]
According to Lemma \ref{lem:F}, for any even number $n\geq 4$ and $k\geq 2$ we have $F(C_n^k) \ge \frac{n^k}{4} $. Since the graph $C_n$ is edge-transitive, and the Cartesian product of edge-transitive graphs is edge-transitive, we can apply the result of Theorem~\ref{thm:edge-transitive} to conclude the result.
\end{proof}

\section{Conclusion} \label{conclusion}
In this paper, we developed an algorithm to generate complete forcing sets for arbitrary graphs. Then, we analyze the size of its output to obtain an upper bound on the complete forcing numbers of graphs. To this end, we relate the size of the output to various parameters of the graph, such as the spectral radius and the degeneracy. Such upper bounds are particularly important for the following reasons. First, these bounds apply to any graph with no limitation on the graph. Second, while the complexity of finding the exact value of the complete forcing number is not known, our bounds can be computed efficiently, i.e., in polynomial-time. We also derive a lower bound on the complete forcing number of edge-transitive graphs in terms of the maximum forcing number of graphs. Several special cases and corollaries of these results have been investigated.

\section{Acknowledgement}
We would like to thank the anonymous reviewers for their valuable comments and insightful suggestions, which have significantly improved the clarity, accuracy, and overall quality of our paper.

This research was in part supported by a grant from IPM (No.1404050011)
\\
\medskip

\noindent \textbf{Declaration of generative AI and AI-assisted technologies in the writing process}

\noindent During the preparation of this work the authors used ChatGPT in order to find and correct typos. After using this tool, the authors reviewed and edited the content as needed and take full responsibility for the content of the publication.
\medskip 

\bibliographystyle{elsarticle-num}
\bibliography{ref} 

@article{zhang2025forcing,
  title={Forcing, anti-forcing, global forcing and complete forcing on perfect matchings of graphs—A survey},
  author={Zhang, Yaxian and He, Xin and Liu, Qianqian and Zhang, Heping},
  journal={Discrete Applied Mathematics},
  volume={376},
  pages={318--347},
  year={2025},
  publisher={Elsevier}
}

@book{west2001,
    title={Introduction to Graph Theory},
    publisher={Prentice Hall},
    author={West, Douglas Brent},
    edition={2nd},
    year={2001}
}

@article{xu2015complete,
  title={Complete forcing numbers of catacondensed hexagonal systems},
  author={Xu, Shou-Jun and Zhang, Heping and Cai, Jinzhuan},
  journal={Journal of Combinatorial Optimization},
  volume={29},
  number={4},
  pages={803--814},
  year={2015},
  publisher={Springer}
}

@article{chan2015linear,
  title={A linear-time algorithm for computing the complete forcing number and the Clar number of catacondensed hexagonal systems},
  author={Chan, Wai Hong and Xu, Shou-Jun and Nong, Ge},
  journal={Match-communications in Mathematical and in Computer Chemistry},
  year={2015},
  volume={74},
  pages={201-216},
  publisher={Citeseer}
}

@article{xu2016coronoid,
  title={Complete forcing numbers of primitive coronoids},
  author={Xu, Shou-Jun and Liu, Xiu-Song and Chan, Wai Hong and Zhang, Heping},
  journal={Journal of Combinatorial Optimization},
  volume={32},
  number={1},
  pages={318--330},
  year={2016},
  publisher={Springer}
}

@article{liu2016complete,
  title={Complete forcing numbers of polyphenyl systems},
  author={Liu, Bingjie and Bian, Hong and Yu, Haizheng},
  journal={Iranian Journal of Mathematical Chemistry},
  volume={7},
  number={1},
  pages={39--46},
  year={2016},
  publisher={University of Kashan}
}

@article{liu2021complete,
  title={Complete forcing numbers of spiro hexagonal systems},
  author={Liu, Bingjie and Bian, Hong and Yu, Haizheng and Li, Jingjing},
  journal={Polycyclic Aromatic Compounds},
  volume={41},
  number={3},
  pages={511--517},
  year={2021},
  publisher={Taylor \& Francis}
}

@article{he2021hexagonal,
  title={Complete forcing numbers of hexagonal systems},
  author={He, Xin and Zhang, Heping},
  journal={Journal of Mathematical Chemistry},
  volume={59},
  number={7},
  pages={1767--1784},
  year={2021},
  publisher={Springer}
}

@article{he2022hexagonal,
  title={Complete forcing numbers of hexagonal systems II},
  author={He, Xin and Zhang, Heping},
  journal={Journal of Mathematical Chemistry},
  volume={60},
  number={4},
  pages={666--680},
  year={2022},
  publisher={Springer}
}

@article{wei2023complete,
  title={Complete forcing numbers of catacondensed phenylene systems},
  author={Wei, Lina and Bian, Hong and Yu, Haizheng and Lin, Guocan},
  journal={Filomat},
  volume={37},
  number={24},
  pages={8309--8317},
  year={2023}
}

@article{chang2021complete,
  title={Complete forcing numbers of rectangular polynominoes},
  author={Chang, Hong and Feng, Y and Bian, Hong and Xu, S},
  journal={Acta Math. Spalatensia},
  volume={1},
  pages={87--96},
  year={2021}
}

@article{he2023complete,
  title={Complete forcing numbers of graphs},
  author={He, Xin and Zhang, Heping},
  journal={ARS Mathematica Contemporanea},
  volume={23},
  number={2},
  pages={P2--09},
  year={2023}
}

@article{he2023completemultipartite,
  title={Complete forcing numbers of complete and almost-complete multipartite graphs},
  author={He, Xin and Zhang, Heping},
  journal={Journal of Combinatorial Optimization},
  volume={46},
  number={2},
  pages={11},
  year={2023},
  publisher={Springer}
}

@article{he2024complete,
  title={Complete forcing numbers of (4, 6)-fullerenes},
  author={He, Xin and Zhang, Heping},
  journal={Discrete Applied Mathematics},
  volume={357},
  pages={385--398},
  year={2024},
  publisher={Elsevier}
}

@book{godsil2001algebraic,
  title={Algebraic Graph Theory},
  author={Godsil, Chris and Royle, Gordon F},
  year={2001},
  volume={207},
  publisher={Springer}
}

@article{yu2004spectral,
  title={On the spectral radius of graphs},
  author={Yu, Aimei and Lu, Mei and Tian, Feng},
  journal={Linear algebra and its applications},
  volume={387},
  pages={41--49},
  year={2004},
  publisher={Elsevier}
}

@article{YUAN1988135,
title = {A bound on the spectral radius of graphs},
author = {Hong Yuan},
journal = {Linear Algebra and its Applications},
volume = {108},
pages = {135-139},
year = {1988}
}

@inproceedings{hayes2006simple,
  title={A simple condition implying rapid mixing of single-site dynamics on spin systems},
  author={Hayes, Thomas P},
  booktitle={2006 47th Annual IEEE Symposium on Foundations of Computer Science (FOCS'06)},
  pages={39--46},
  year={2006},
  organization={IEEE}
}

@book{bickle2010cores,
    title = {The k-cores of a graph},
    publisher={Western Michigan University},
    author = {Bickle, Allan},
    year={2010},
    url={https://books.google.com/books?id=MJ_mZwEACAAJ}
}

@article{hickingbotham2023structural,
  title={Structural properties of graph products},
  author={Hickingbotham, Robert and Wood, David R},
  journal={Journal of Graph Theory},
  year={2023},
  publisher={Wiley Online Library}
}

@article{dvovrak2010spectral,
  title={Spectral radius of finite and infinite planar graphs and of graphs of bounded genus},
  author={Dvo{\v{r}}{\'a}k, Zden{\v{e}}k and Mohar, Bojan},
  journal={Journal of Combinatorial Theory, Series B},
  volume={100},
  number={6},
  pages={729--739},
  year={2010},
  publisher={Elsevier}
}

@article{adams2004forced,
  title={On the forced matching numbers of bipartite graphs},
  author={Adams, Peter and Mahdian, Mohammad and Mahmoodian, Ebadollah S},
  journal={Discrete Mathematics},
  volume={281},
  number={1-3},
  pages={1--12},
  year={2004},
  publisher={Elsevier}
}

@article{riddle2002minimum,
  title={The minimum forcing number for the torus and hypercube},
  author={Riddle, Matthew E},
  journal={Discrete mathematics},
  volume={245},
  number={1-3},
  pages={283--292},
  year={2002},
  publisher={Elsevier}
}

@article{ebrahimi2023maximum,
  title={Maximum fractional forcing number of the powers of even cycles},
  author={Ebrahimi, Javad B and Karimy, Sajed and Ranjbarzadeh, Shayan},
  journal={Advanced Studies: Euro-Tbilisi Mathematical Journal},
  volume={16},
  number={1},
  pages={1--11},
  year={2023},
  publisher={Tbilisi Centre for Mathematical Sciences}
}

@article{ebrahimi2020fractional,
title = {Fractional forcing number of graphs},
journal = {Discrete Mathematics},
volume = {349},
number = {2},
pages = {114739},
year = {2026},
author = {Javad {B. Ebrahimi} and Babak Ghanbari},
}

@article{pachter1998forcing,
  title={Forcing matchings on square grids},
  author={Pachter, Lior and Kim, Peter},
  journal={Discrete Mathematics},
  volume={190},
  number={1-3},
  pages={287--294},
  year={1998},
  publisher={Elsevier}
}

\end{document}